\documentclass[onecolumn,final]{svjour3}  
\usepackage{verbatim}
\usepackage{graphicx,multirow,booktabs,subfig}
\usepackage{amsmath,amssymb,bm}
\usepackage{cases}
\usepackage[T1]{fontenc}
\usepackage{hyperref}
\allowdisplaybreaks
\hypersetup{
  colorlinks =true,
  linkcolor = blue,
  filecolor = blue,
  citecolor = magenta,  
  urlcolor = magenta,
  }
  \usepackage[a4paper,
  bindingoffset=0.2in,
  left=1in,
  right=1in,
  top=1in,
  bottom=1in,
  footskip=.25in]{geometry}
\usepackage{setspace}
\usepackage{siunitx}
\usepackage{csvsimple}
\usepackage{url}

\usepackage{mathalpha}
\usepackage{mathtools} 
\usepackage{extarrows} 

\usepackage[style=authoryear,backref=true]{biblatex}
\usepackage{enumitem, xspace}

\usepackage{accents}
\usepackage{xfrac}

\renewcommand{\bf}[1]{\mathbf{#1}}

\DeclareMathOperator*{\argmax}{arg\,max}

\begin{document}

\title{A Parametric, Second-Order Cone Representable Model of Fairness for Decision-Making Problems}

\titlerunning{SOC Representable Model of Fairness}        

\author{Kaarthik Sundar \and Deepjyoti Deka \and Russell Bent}


\institute{K. Sundar (corresponding author) \at
        Information Systems \& Modeling, \\
              Los Alamos National Laboratory, \\ 
              Los Alamos, New Mexico, USA \\
              \email{\texttt{kaarthik@lanl.gov}} \\ 
        \\ 
       D. Deka\at
        MIT Energy Initiative, \\
              MIT,Cambridge, Massachusetts, USA \\
              \email{\texttt{deepj87@mit.edu}} \\ 
        \\ 
        R. Bent \at 
        Applied Mathematics \& Plasma Physics, \\ 
        Los Alamos National Laboratory, \\ 
        Los Alamos, New Mexico, USA \\
        \email{\texttt{rbent@lanl.gov}}  
}

\date{}

\maketitle  
\begin{abstract}
The article develops a parametric model of fairness called ``$\varepsilon$-fairness'' that can be represented using a single second-order cone constraint and incorporated into existing decision-making problem formulations without impacting the complexity of solution techniques. We develop the model from the fundamental result of finite-dimensional norm equivalence in linear algebra and show that this model has a closed-form relationship to an existing metric for measuring fairness widely used in the literature. Finally, a simple case study on the optimal operation of a damaged power transmission network illustrates its effectiveness.
\keywords{fairness \and second-order cone \and parametric model \and optimization}
\end{abstract}

\section{Introduction} \label{sec:intro}
Decision-making problems are optimization problems that arise in various engineering applications, like communication networks, vehicle routing, power grid operations, etc. Traditionally these optimization problems strive for efficiency by minimizing total cost or maximizing total benefit. Depending on the application, the cost can correspond to travel distance, damage, or control effort; the benefit can correspond to throughput, revenue, etc. One common feature of all these problems is that the benefit or cost is usually spread across multiple agents or entities. It is intuitively easy to see that the optimal solution to the decision-making problems that strive for efficiency may lead to an ``unfair'' distribution of benefits or costs across different agents, i.e., a significant portion of the benefits or costs is limited to a subset of agents. Nevertheless, it is also not difficult to realize that any attempt to obtain a fair distribution across agents will necessarily come at the expense of decreasing the overall efficiency. Thus, there inherently is a trade-off between efficiency and fairness that the decision-maker must make to ensure the equitable distribution of benefit or cost. This paper seeks to mathematically quantify this trade-off by developing a parametric model for incorporating fairness that can directly be embedded into existing optimization problem formulations that arise in various applications.

Modeling and incorporating fairness into optimization problems is not new in the optimization literature. Examining the effect these models have on the solutions of the underlying optimization problem has been studied in a variety of settings ranging from communication networks (see \cite{bertsekas2021data,shakkottai2008network}) and financial applications (see \cite{stubbs2009multi}) to general combinatorial optimization problems (see \textcite{bektacs2020using} and references therein). Before we present an overview of the existing literature, we set up the notations along the lines of recent work in \textcite{xinying2023guide}. To that end, we assume the underlying optimization problem with $n$ agents takes the form: 
\begin{gather}
  \max ~ \{ f(\bm x) : \bm x \in \mathcal X\} \label{eq:opt}
\end{gather}
where, $\bm x \in \mathbb R^m$ is the vector of decision variables, $f(\bm x)$ measures the efficiency of the solution $\bm x$. In addition, we consider a $n$-dimensional vector $\bm u$ of \emph{non-negative utilities} assigned to the $n$ agents, respectively, defined using a vector-valued linear map $\bm W: \mathcal X \rightarrow \mathbb R_{\geqslant 0}^n$ that changes with $\bm x$. Enforcing fairness now corresponds to ensuring the fairness of the distribution of utilities $\bm u$; depending on the application, the utilities can correspond to profit, negative cost, throughput, etc. With this setup, existing literature on incorporating fairness concerns with the design of a scalar-valued concave utility function, denoted by $U(\bm u)$, that aggregates the vector $\bm u$ of utilities into $U(\bm u)$ to represent the desirability/fairness of utility distribution.\footnote{In this article, we focus only on the literature that concerns parametric utility functions that quantify the inherent trade-off between fairness and efficiency. For a general survey of other non-parametric utility functions, interested readers are referred to \textcite{xinying2023guide}.} The following convex optimization problem is then solved to enforce a certain level of fairness using the utility function $U(\bm u)$ :
\begin{gather}
 \max ~ \{U(\bm u) : \bm u = \bm W(\bm x), \bm x \in \mathcal X\} \label{eq:utility}
\end{gather}
The main challenge then is to design a function $U(\bm u)$ so that when \eqref{eq:utility} is solved to optimality, the optimal utility distribution would be qualitatively the fairest for that particular $U(\bm u)$. Existing work, instead of designing a single function $U(\bm u)$ seeks to design a parametric family of utility functions that enable empirically studying the trade-off between efficiency and fairness.

The literature's first widely used parametric utility function is ``$\alpha$-fairness'' (\cite{mo2000fair}). Here, the utility function is parameterized by $\alpha$ and is given by 
\begin{gather}
 U_{\alpha}(\bm u) = 
 \begin{cases}
  \frac 1{1-\alpha} \sum_{1\leqslant i \leqslant n} u_i^{1-\alpha} & \text{ for $\alpha \geqslant 0$ and $\alpha \neq 1$}, \\ 
  \sum_{1\leqslant i \leqslant n} \log(u_i) & \text{ for $\alpha = 1$}.
 \end{cases} \label{eq:alpha-fair}
\end{gather}
$U_{\alpha}(\bm u)$ was introduced initially in the context of network congestion control by \textcite{mo2000fair}. This utility function is widely used in communication networks for resource allocation (\cite{shakkottai2008network}) and, more recently, in applications like air traffic management (\cite{yu2023alpha}). When $\alpha =0$, $U_{\alpha}(\bm u)$ reduces to the sum of the individual utilities, i.e., $\sum_{1\leqslant i \leqslant n} u_i$, that is referred to as the ``utilitarian'' criteria (\cite{xinying2023guide}). When $\alpha = \infty$, it reduces to the minimum utility value in the vector $\bm u$, i.e., $\min_{1\leqslant i \leqslant n} u_i$; this utility function is referred to as the ``max-min'' criteria for enforcing fairness and was initially developed in political philosophy by \textcite{Rawls1971}. The case when $\alpha = 1$ is known as ``proportional fairness'' was introduced in the context of communication networks in \textcite{kelly1997charging}. Nash studied this concept as a solution for players bargaining over the allocation of a shared resource (\cite{nash1950bargaining}); hence, it is also referred to as the ``Nash bargaining solution'' in economics. The general critique of the utility function $U_{\alpha}(\bm u)$ associated with $\alpha$-fairness is that despite $\alpha$ quantifying efficiency-fairness trade-offs (see \cite{xinying2023guide}), it is often not apparent apriori for a particular trade-off, what value of $\alpha$ is appropriate for the application.

The second widely used parametric utility function is called ``$p$-norm fairness'' and is given by the negative of the $\ell^p$ norm (a concave function), as
\begin{align}
 U_p(\bm u) = -\|\bm u\|_p \text{ for } p \geqslant 1. \label{eq:norm-fair}
\end{align} 
This utility function is widely used in the operations research literature for applications like vehicle routing, facility location, and nurse rostering (\cite{bektacs2020using,barbati2016equality,matl2018workload,martin2013cooperative}). Note that the ``min-max'' and ``utilitarian'' criteria for enforcing fairness correspond to $p$ values of $\infty$ and $1$, respectively. Existing literature (\cite{bektacs2020using}) interprets the value of $p$ to quantify the trade-off between efficiency and fairness, but there is no straightforward technique to identify the appropriate value of $p$ for a given application and its impact on efficiency. 

For both $\alpha$-fairness and $p$-norm fairness, existing work uses fairness indices to quantify how fair the solutions obtained by maximizing the corresponding utility functions are. To the best of our knowledge, no direct or indirect relationship exists between either of the utility functions in \eqref{eq:alpha-fair} or \eqref{eq:norm-fair} and any of the fairness indices used in the literature (\cite{xinying2023guide}). Hence, the general strategy used in the literature is to solve \eqref{eq:utility} with $U_{\alpha}(\bm u)$ ($U_p(\bm u)$) for a subset of $\alpha$ ($p$) values, compute the fairness indices a posteriori using the optimal $\bm u$ and check if they provide a satisfactory trade-off between fairness and efficiency for the application. Efficiency is measured by evaluating $f(\bm x)$ -- the objective function in Eq. \eqref{eq:opt}.
One of the most widely used (see \cite{shakkottai2008network, schwarz2011throughput, guo2013jain, guo2014throughput, masoumi2023fairness, kumari2023performance}) fairness indices to quantify fairness of utility vector $\bm u$ is ``Jain et al.'' (\cite{Jain1984}) index, introduced again in the context of communication networks, and given by 
\begin{gather}
 \text{Jain et al. Index: } \mathrm{JI}(\bm u) \triangleq \frac 1n \cdot \frac{\left( \sum_{1\leqslant i \leqslant n} u_i\right)^2}{\sum_{1\leqslant i \leqslant n} u_i^2}. \label{eq:ji}
\end{gather}
$\mathrm{JI}(\bm u)$ takes the value $1/n$ when exactly one element of $\bm u$ is non-zero and $1$ when all the elements of $\bm u$ are equal. Note that the denominator in \eqref{eq:ji} is the square of the 2-norm of $\bm u$. Throughout the rest of the presentation, we shall use the Jain et al. index to measure the fairness of $\bm u$'s distribution. Interested readers are referred to \textcite{xinying2023guide} for other commonly used indices like the Gini index (\cite{Gini1936}), Robin Hood index (\cite{kennedy1996income}) etc.

\subsection{Contributions} To overcome the non-trivial task of choosing parameters $\alpha$ and $p$ in existing fairness models to get the desired efficiency-fairness trade-off, this article develops a novel, parametric, and convex model for fairness called \emph{$\varepsilon$-fairness} starting from the equivalence of finite-dimensional norms. The parameter $\varepsilon \in [0, 1]$ indicates the fairness level with $\varepsilon = 0$ and $\varepsilon = 1$, corresponding to an unfair and fair distribution of $\bm u$. Crucially, we show that $\varepsilon$-fairness can be incorporated as a \emph{second-order cone (SOC) constraint} - \textbf{$g(\bm u, \varepsilon) \leqslant 0$}\footnote{for a fixed value of $\varepsilon$}, within the optimization problem \eqref{eq:opt}. The modified problem given in \eqref{eq:contribution}
\begin{gather}
 z(\varepsilon) = \max ~ \{f(\bm x) : \bm u = \bm W(\bm x), g(\bm u, \varepsilon) \leqslant 0, \bm x \in \mathcal X\} \label{eq:contribution}
\end{gather}
directly enforces a certain level of fairness on the distribution of $\bm u$. Note that since the SOC constraint is convex, it ensures no additional algorithmic complexity is required to solve the underlying optimization problem, i.e., if \eqref{eq:opt} is convex, \eqref{eq:contribution} will continue to be convex. We then show that the existing Jain et al. index can be expressed as a function of $\varepsilon$ in closed form, which will also result in a systematic approach to choosing $\varepsilon$ for the application in question. Furthermore, we shall present beneficial theoretical properties like monotonicity of $z(\varepsilon)$ for $\varepsilon$-fairness, which allow for easy quantification of the trade-off between efficiency and fairness and the development of approaches to select an appropriate value for $\varepsilon$.

The article is structured as follows: After developing the model of $\varepsilon$-fairness from in Sec. \ref{sec:model}, we present the theoretical results that relate $\varepsilon$-fairness with the Jain et al. index and the properties of $\varepsilon$-fairness in Sec. \ref{sec:properties} followed by a case study on a power system operation problem in Sec. \ref{sec:case-study} and concluding with Sec. \ref{sec:conclusion}.

\section{Development of $\varepsilon$-fairness} \label{sec:model}
We start by invoking a well-known inequality that relates $\|\bm u\|_1$ and $\|\bm u\|_2$ for any $\bm u \in \mathbb R_{\geqslant 0}^n$, derived using the fundamental result of ``equivalence of norms'' (see \cite{horn2012matrix})
\begin{gather}
 \|\bm u \|_2 \leqslant \| \bm u\|_1 \leqslant \sqrt{n} \cdot \|\bm u\|_2. \label{eq:norm-eq}
\end{gather} 
In \eqref{eq:norm-eq}, it is not difficult to see that
\begin{enumerate}[label=(\roman*)]
 \item the first inequality, holds at equality, i.e., $\|\bm u \|_2 = \| \bm u\|_1$, when $u_i$ is zero for all but one component (\textit{most unfair}), and 
 \item the second inequality in \eqref{eq:norm-eq} holds at equality, i.e., $\sqrt{n} \cdot \|\bm u\|_2 = \| \bm u\|_1$, when every component of $\bm u$ is equal (\textit{most fair}),
\end{enumerate}
We use the intuition provided by (i) and (ii) to introduce the following parametric definition of fairness: 
\begin{definition} \label{def:eps-fairness} \textit{$\varepsilon$-fairness} -- For any $\bm u \in \mathbb R_{\geqslant 0}^n$ and $\varepsilon \in [0, 1]$, we say $\bm u$ is ``$\varepsilon$-fair'', if $\left(1-\varepsilon + \varepsilon\sqrt{n} \right) \cdot \|\bm u \|_2 = \|\bm u\|_1$. Here, we say $\bm u$ is most unfair (fair) when $\varepsilon = 0$ ($\varepsilon = 1$), respectively. Indeed, at that value, $\bm u$ attains the lower (upper) bound in \eqref{eq:norm-eq}. Furthermore, we define $\bm u$ to be ``\textit{at least $\varepsilon$-fair}'', if 
\begin{gather}
 \left(1-\varepsilon + \varepsilon\sqrt{n} \right) \cdot \|\bm u \|_2 \leqslant \|\bm u\|_1. \label{eq:soc-e-fair}
\end{gather}
\end{definition}
We remark that $\left(1-\varepsilon + \varepsilon\sqrt{n} \right)$ in Definition \ref{def:eps-fairness} is the convex combination of $\sqrt n$ and $1$, with $\varepsilon$ as the multiplier. The definition also allows modification of the optimization problem in \eqref{eq:opt} to restrict the feasible set of utility vectors to be at least $\varepsilon$-fair, by adding \eqref{eq:soc-e-fair} as a constraint, as presented in \eqref{eq:contribution}, i.e., 
\begin{gather}
 g(\bm u, \varepsilon) \leqslant 0 \equiv \left(1-\varepsilon + \varepsilon\sqrt{n} \right) \cdot \|\bm u \|_2 \leqslant \|\bm u\|_1.
\end{gather}
When $\varepsilon$ is set to $0$ the \eqref{eq:soc-e-fair} enforces the trivial bound on the 1-norm of $\bm u$, i.e., $\| \bm u \|_2 \leqslant \| \bm u \|_1$. In the next section, we present some theoretical results that will provide intuition on how $\varepsilon$-fairness enforces the distribution of utilities to be fair.

\section{Theoretical properties} \label{sec:properties}
We start by introducing some notations. Let $\mu$ and $\sigma^2$ denote the sample mean and variance of all the utility values in the utility vector $\bm u$. Then, combining the definitions of sample mean, sample variance, and the assumption that all utility values are non-negative, we have
\begin{gather}
 \mu = \frac{\|\bm u\|_1}n \text{ and } \sigma^2 = \frac 1{n-1}\left( \|\bm u \|_2^2 - n \mu^2 \right). \label{eq:sample-defn}
\end{gather}
From \eqref{eq:soc-e-fair}, enforcing $\bm u$ to be at least $\varepsilon$-fair is equivalent to
\begin{gather}
 \| \bm u \|_2^2 \leqslant \frac{\| \bm u \|_1^2}{\left(1-\varepsilon + \varepsilon\sqrt{n} \right)^2}\nonumber\\
 \Rightarrow~\quad \sigma^2 \leqslant \frac{n}{n-1} \left( \frac n{\left(1-\varepsilon + \varepsilon\sqrt{n} \right)^2} - 1\right) \mu^2, \label{eq:variance-bound}
\end{gather}
where \eqref{eq:variance-bound} follows from definitions in \eqref{eq:sample-defn}. If we define 
\begin{gather}
 h(\varepsilon) \triangleq \frac{n}{n-1} \left( \frac n{\left(1-\varepsilon + \varepsilon\sqrt{n} \right)^2} - 1\right),
\end{gather}
then \eqref{eq:variance-bound} reduces to 
\begin{gather}
 c_v^2 \triangleq \left(\frac{\sigma}{\mu}\right)^2 \leqslant h(\varepsilon) \label{eq:cv}
\end{gather}
where $c_v$ is the ``coefficient of variation'' in statistics (\cite{everitt2010cambridge}), a standardized measure of frequency distribution's dispersion. It is not difficult to see that $h(\varepsilon)$ is strictly decreasing with $\varepsilon$, indicating that as $\varepsilon$ is varied from $0 \rightarrow 1$, the upper bound on $c_v$ decreases from a trivial value of $\sqrt n$ (always true) to $0$ (enforces the utility values for all agents to be equal). In other words, $\varepsilon$-fairness enforces fairness by controlling the dispersion in $\bm u$'s distribution through $\varepsilon$. This observation leads to a simple bisection on $\varepsilon \in [0, 1]$ providing the value of $\varepsilon$ that can result in the required fairness-efficiency trade-off. To the best of our knowledge, such algorithms do not exist and, in general, are difficult to design for the utility functions $U_{\alpha}(\bm u)$ and $U_p(\bm u)$ for varying $\alpha$ and $p$. 

Next, we present a key result that enables us to identify trends in the feasibility and optimality of \eqref{eq:contribution}.
\begin{proposition} \label{prop:eps-max}
Given $\bar\varepsilon \in [0, 1]$, if the problem \eqref{eq:contribution} is infeasible, then it is also infeasible for any $\varepsilon \in [\bar \varepsilon, 1]$.
\end{proposition}
\begin{proof}
 To prove this statement, we let 
 \begin{gather*}
 \mathcal Y^{\varepsilon} \triangleq \left\{ (\bm u, \bm x)~ \left\vert ~
 \begin{gathered}
  \bm u = \bm W(\bm x), ~\bm x \in \mathcal X, \text{ and } \\ 
  \left(1-\varepsilon + \varepsilon\sqrt{n} \right) \cdot \| \bm u \|_2 \leqslant \| \bm u \|_1
 \end{gathered} \right. \right\}. 
 \end{gather*}
 $\mathcal Y^{\varepsilon}$ is the set of feasible solutions for \eqref{eq:contribution}. The proof now follows from the observation that for any $\varepsilon_1, \varepsilon_2 \in [0, 1]$ such that $\varepsilon_1 \geqslant \varepsilon_2$, $\left(1-\varepsilon_1 + \varepsilon_1\sqrt{n} \right) \geqslant \left(1-\varepsilon_2 + \varepsilon_2\sqrt{n} \right)$ which gives $\mathcal Y^{\varepsilon_1} \subseteq \mathcal Y^{\varepsilon_2}$.  \qed
\end{proof} 
The above proposition states that the size of the feasible set of solutions to \eqref{eq:contribution} decreases as we increase $\varepsilon$. In other words, if the distribution of utilities in $\bm u$ cannot be made $\bar\varepsilon$-fair, it cannot be made any ``fair-er'', aligning with our intuitive understanding of fairness. A practical consequence of the proposition is that there exists a $0\leqslant \varepsilon^{\max}\leqslant 1$ such that for any $\varepsilon \in (\varepsilon^{\max}, 1]$, problem \eqref{eq:contribution} is infeasible, and $\varepsilon \in (0,\varepsilon^{\max}]$, problem \eqref{eq:contribution} is feasible.
From here on, we shall refer to $[0, \varepsilon^{\max}]$ as the ``feasibility domain'' for \eqref{eq:contribution}. 
We now state another useful result in quantifying the trade-off between fairness and efficiency in optimization problems. 
\begin{proposition} \label{prop:monotonicity}
The univariate function $z(\varepsilon)$, defined in \eqref{eq:contribution}, monotonically decreases with $\varepsilon$ for all values in the feasibility domain of \eqref{eq:contribution}. 
\end{proposition}
\begin{proof}
 The proof follows from the following two observations:
 \begin{enumerate}
  \item the feasible set of solutions of \eqref{eq:contribution} becomes smaller as $\varepsilon$ increases within its feasibility domain (see proposition \ref{prop:eps-max}) and
  \item the problem \eqref{eq:contribution} is a maximization problem.
 \end{enumerate}
 Together, the above observations lead to the conclusion that if $\varepsilon_1 \geqslant \varepsilon_2$, then $z(\varepsilon_1) \leqslant z(\varepsilon_2)$. \qed
\end{proof}
Proposition \ref{prop:monotonicity} tells us that when $\varepsilon$ increases, the distribution of $\bm u$ becomes fairer but the maximum achievable efficiency, given by $z(\varepsilon)$, decreases. This conforms with our intuitive understanding of the trade-off between fairness and efficiency and enables quantifying this trade-off for different values of $\varepsilon$. As emphasized in the introduction, we reiterate that monotonicity property analogous to Proposition \ref{prop:monotonicity} and systematic techniques to quantify the trade-off between efficiency and fairness are lacking in the state-of-the-art parametric utility-based models of fairness.

\subsection{Relationship to Jain et al. index} \label{subsec:JI}
We now show that enforcing $\varepsilon$-fairness of the utilities is equivalent to setting the Jain et al. index of the utilities to a bijective function of $\varepsilon$. This will translate any value of the Jain et al. index sought to an equivalent $\varepsilon$ value to enforce $\varepsilon$-fairness. To the best of our knowledge, this cannot be done using any of the existing models of fairness.
\begin{proposition} \label{prop:ji}
Enforcing $\bm u$ to be $\varepsilon$-fair is equivalent to setting $$\mathrm{JI}(\bm u) = \frac{(1-\varepsilon+\varepsilon\sqrt n)^2}n$$
\end{proposition}
\begin{proof}
 Utilizing the non-negativity of the utility values, Jain et al. index in \eqref{eq:ji} can be equivalently rewritten as $$\mathrm{JI}(\bm u) = \frac{\|\bm u\|_1^2}{n\cdot \|\bm u\|_2^2}.$$ Combining the above equation with the definition of enforcing $\varepsilon$-fairness of $\bm u$ leads to
 \begin{gather}
  \mathrm{JI}(\bm u) = \frac{\left(1-\varepsilon + \varepsilon\sqrt{n} \right)^2}{n} \triangleq w(\varepsilon). \label{eq:w-defn}
 \end{gather}
It is also easy to observe that enforcing $\bm u$ to be at least $\varepsilon$-fair is equivalent to $\mathrm{JI}(\bm u) \geqslant w(\varepsilon)$.  \qed
\end{proof}
Proposition \ref{prop:ji} provides a bijective function, $w(\varepsilon)$, that can translate any value of $\varepsilon$ to an equivalent value of the Jain et al. index as also observed in Fig. \ref{fig:ji-eps}. The proposition shows that the $\varepsilon$-fairness model is a scaled version of the Jain et al. index, i.e., the following statement holds for any $\bm u \in \mathbb R^n{\geqslant 0}$:
\begin{gather*}
    \left(1-\varepsilon + \varepsilon\sqrt{n} \right) \cdot \|\bm u \|_2 \leqslant \|\bm u\|_1  ~\xLongleftrightarrow{\text{Definition \ref{def:eps-fairness}}}~ \bm u \text{ is at least $\varepsilon$-fair } \\
    \xLongleftrightarrow{\text{Proposition \ref{prop:ji}}}~ \mathrm{JI}(\bm u) \geqslant w(\varepsilon) 
    ~\xLongleftrightarrow{\text{Eq. \eqref{eq:ji} and take sq. root}}~ \sqrt{n \cdot w(\varepsilon)} \cdot \|\bm u\|_2 \leqslant \| \bm u \|_1
\end{gather*}
Hence, a practitioner can use either $\varepsilon$-fairness or $\mathrm{JI}(\bm u) \geqslant w(\varepsilon)$ to enforce fairness and the above statement guarantees that both approaches of enforcing fairness will yield the same results. \textcite{Jain1984} hypothesize the formula for the Jain et al. index to be given by Eq. \eqref{eq:ji} and show that it has desirable properties of population size independence, scale and metric independence, boundedness, and continuity with respect to small changes in $\bm u$ without providing any insight as to how one can derive the formula and mathematical quantification as to why it measures fairness. The development of $\varepsilon$-fairness in Sec. \ref{sec:model} starting from a fundamental result in linear algebra provides this much-needed foundation to the Jain et al. index in Eq. \eqref{eq:ji}. Furthermore, we remark that all the results proved in Sec. \ref{sec:properties} can be extended to use the Jain et al. index by replacing $\left(1-\varepsilon + \varepsilon\sqrt{n} \right)$ with $\sqrt{n \cdot w(\varepsilon)}$. 

\begin{figure}
 \centering
 \includegraphics[scale=0.8]{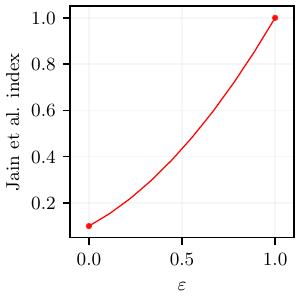}
 \caption{Plot of $\mathrm{JI}(\bm u) = w(\varepsilon)$ when $n = 10$.}
 \label{fig:ji-eps}
\end{figure}

\section{Case study} \label{sec:case-study}
The case studies in this section are specifically designed to answer the following questions: (i) How can $\varepsilon$-fairness be integrated into existing optimization problems, (ii) How can we both qualitatively and quantitatively illustrate that fairness is being enforced, and (iii) How can we quantify the trade-off between efficiency and fairness? To that end, we choose the application of electric transmission network operations. This is motivated by the rising interest in incorporating fairness into power system optimization to ensure social, energy justice, and health impacts of grid operations are given their due importance (\cite{Mathieu2023,taylor2023managing,huang2023inequalities,hashmi2022can,wang2023local}). In particular, we choose the Minimum Load Shedding (MLS) problem formulated in \textcite{coffrin2018relaxations} with a Direct Current (DC) power flow model that focuses on minimizing the sum of the load shed to each customer after an extreme event damages a subset of components of the network. We incorporate $\varepsilon$-fairness into the MLS problem and answer the questions posed at the beginning of this section using the theoretical results presented thus far. We start with the formulation of the MLS problem.

\subsection{Problem formulation} \label{subsec:formulation}
Consider a damaged power network with undamaged transmission lines $\mathcal L$, buses $\mathcal B$, generators $\mathcal G$, and loads (demands) $\mathcal D$. We also let $\mathcal G_b$ and $\mathcal D_b$ denote the subset of generators and loads co-located on bus $b \in \mathcal B$. In an electric power network, power is produced at the generators $\mathcal G$ and consumed at the loads or demand locations $\mathcal D$. The power injected into the system is transmitted using transmission lines; these lines, in turn, have capacity limits (also referred to as thermal limits) indicative of the maximum amount of power that can flow through the line. Furthermore, each bus in $\mathcal B$ is also associated with a voltage, a phasor quantity, i.e., it has a magnitude and a phase angle. For a DC approximation of the physics of power flowing a network, the voltage magnitude is assumed to be unity at all buses; the power flowing the line that connects bus $i$ and $j$ is proportional to the phase angle difference at the two buses (with the proportionality constant given by the susceptance of the line). Armed with this brief overview of electric power networks, we proceed by introducing additional notations as follows: Let the decision variables $\theta_b$ denote the phase angle at bus $b \in \mathcal B$, $p_g$ denote the power generated by generator $g \in \mathcal G$, $p_{ij}$ denote the power flowing on the transmission line $(i, j) \in \mathcal L$ and non-negative variable $d_i$ indicate the amount of load shed at $i \in \mathcal D$. Let parameters $p_{ij}^{\max}, b_{ij}$ be the respective thermal limit and susceptance for line $(i, j) \in \mathcal L$, and $p_g^{\min}, p_g^{\max}$ be the generation lower and upper limits $ ~\forall g \in \mathcal G$. We use $d_i^{\max} ~\forall i \in \mathcal D$ to represent the total power demands. The MLS problem is formulated as:
\begin{gather*}
 \min \sum_{i \in \mathcal D} d_i \text{ subject to: } \\ 
 \mathcal X_{\mathrm{DC}} \triangleq \left\{ 
 \begin{gathered}
 p_{ij} = b_{ij} (\theta_i - \theta_j) ~~\quad \forall (i, j) \in \mathcal L \\ 
 \sum_{g \in \mathcal G_b} p_g - \sum_{i \in \mathcal D_{b}} (d_i^{\max} - d_i) = \sum_{j \in \mathcal B} p_{bj} - \sum_{j \in \mathcal B} p_{jb} ~~ \forall b \in \mathcal B \\ 
 p_g \in [p_g^{\min}, p_g^{\max}] ~~\quad \forall g \in \mathcal G \\
 d_i \in [0, d_i^{\max}]  ~~\quad \forall i \in \mathcal D\\ 
 -p_{ij}^{\max} \leqslant p_{ij} \leqslant p_{ij}^{\max} ~~\quad \forall (i, j) \in \mathcal L
 \end{gathered}
 \right\} 
\end{gather*}
In the above definition of $\mathcal X_{\mathrm{DC}}$, the first two equations are the power flow for each line $(i, j) \in \mathcal L$ and the nodal balance at each bus $b \in \mathcal B$. The remaining inequality constraints are the capacity limits. If we let $\bm d \in \mathbb{R}_{\geqslant 0}$ denote the vector of the individual load sheds $d_i$, $i \in \mathcal D$, and $\bm e$ denote the vector of $1$s, we can rewrite the above formulation on the lines of \eqref{eq:opt} as 
\begin{gather}
 \min ~\{ \bm e^{\intercal} \bm d : \bm d \in \mathcal X_{\mathrm{DC}} \}. \label{eq:opt-mls}
\end{gather}
In \eqref{eq:opt-mls}, the total load shed, $\bm e^{\intercal} \bm d$, is inversely proportional to the efficiency of the system, i.e., the lower the load shed, the higher the efficiency of the system. In this setting, an operator seeks to be fair to all the grid customers by ensuring the $\bm d$ distribution is fair. We enforce this condition using the developed model for $\varepsilon$-fairness, analogous to \eqref{eq:contribution}, as 
\begin{gather}
 z(\varepsilon) = \min\left\{ \bm e^{\intercal} \bm d \left\vert 
 \bm d \in \mathcal X_{\mathrm{dc}} \text{ and } 
 \left(1-\varepsilon + \varepsilon\sqrt{|\mathcal D|} \right) \| \bm d \|_2 \leqslant \| \bm d \|_1 
 \right. \right\}. \label{eq:fair-mls}
\end{gather}

Furthermore, to compare the solutions obtained from enforcing $\varepsilon$-fairness to that obtained using the utility function-based approaches, we formulate and solve the following problem that seeks to maximize the utility function $U_p(\bm u)$ for different values of $p \geqslant 1$:
\begin{gather}
    \bm d_p^* = \argmax \left\{ U_p(\bm d) \left\vert \bm d \in \mathcal X_{\mathrm{dc}} \right. \right\} \text{ and } z_p^* \triangleq \bm e^{\intercal} \bm d_p^* \label{eq:p-norm-mls}
\end{gather}
Though the same formulation extends for the utility function for $\alpha$-fairness, we do not present results with $U_{\alpha}(\bm d)$ as they were similar to those that were obtained using $U_p(\bm d)$. Finally, in Eq. \eqref{eq:p-norm-mls}, $z_p^*$ denotes the sum of the individual components of $\bm d_p^*$. 

\subsection{Network data and computational platform} \label{subsec:data}
All the case studies are performed on the Power Grid Library's (\cite{pglib}) IEEE 14-bus test system\footnote{The choice of a small test network is motivated by the ease of visualizing results. We reiterate that enforcing $\varepsilon$-fairness to existing optimization problems involves the addition of a single convex SOC constraint. If the original optimization problem is convex, adding the SOC constraint does not change the computational complexity needed to find the optimal solution. This was observed empirically when repeating all the experiments presented in this section on a more extensive IEEE RTS-96 test network.} with 11 customers (loads).
For the 14-bus network, random damage scenarios were generated, each containing five damaged lines that caused non-zero total load shed for the MLS problem \eqref{eq:opt-mls}. This resulted in 9765 scenarios on which the MLS problem and its fair variant were tested. 
The formulations presented in this article were implemented in the Julia programming language (\cite{bezanson2017julia}) using the JuMP (\cite{dunning2017jump}) package. All linear and SOC optimization problems were solved using CPLEX (\cite{cplex}), a commercial linear and SOC programming solver. All experiments were run on an Intel Haswell 2.6 GHz, 62 GB, 20-core machine at Los Alamos National Laboratory.

\subsection{Distribution of load shed for varying $\varepsilon$}
Here, we illustrate graphically what enforcing $\varepsilon$-fairness means for the MLS problem in terms of the distribution of $\bm d$. To that end, we let $\varepsilon$ vary in the set $\{0.0, 0.3, 0.6, 0.9\}$, solve \eqref{eq:fair-mls} with that value of $\varepsilon$ for each of the 9765 damage scenarios. The box plot of the load sheds for all scenarios in Fig. \ref{fig:ls-q2} shows that fairness gets imposed when $\varepsilon$ is increased. This type of qualitative analysis can also be done with existing parametric utility function-based methods ($U_\alpha(\bm u)$ and $U_p(\bm u)$) with one important caveat: fairness is not monotonic with the parameters ($\alpha$ and $p$) in those functions; results demonstrating this failure in monotonicity is presented in the subsequent sections.

\begin{figure*}[htbp]
 \centering
 \includegraphics[scale=0.8]{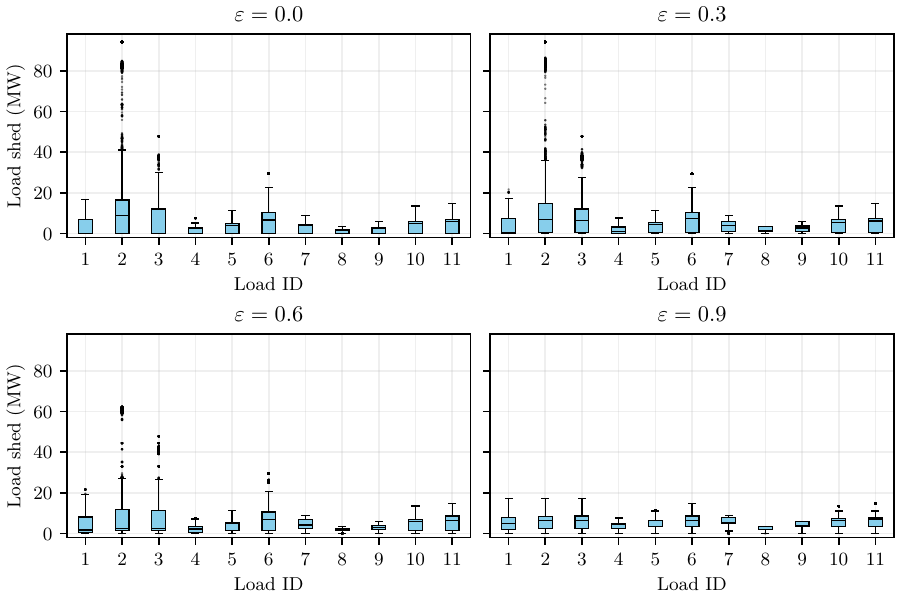}
 \caption{Box plot of load shed obtained by solving the fair version of the MLS problem \eqref{eq:fair-mls} for different values of $\varepsilon$. Notice that when $\varepsilon = 0.0$, \eqref{eq:fair-mls} and \eqref{eq:opt-mls} are equivalent.}
 \label{fig:ls-q2}
\end{figure*}

\subsection{When enforcing $\varepsilon$-fairness is infeasible}
Table \ref{tab:infeasible-q2} shows the number of damage scenarios for which \eqref{eq:fair-mls} is infeasible for different values of $\varepsilon$. The number of infeasible scenarios increases with $\varepsilon$, empirically validating the correctness of Proposition \ref{prop:eps-max}. Note that if for a given $\varepsilon \in [0, 1]$, no $\varepsilon$-fair solution exists, the value of $\varepsilon$ has to be reduced to ensure $\varepsilon < \varepsilon^{\max}$; here the value of $\varepsilon^{\max}$ can be computed systematically using bisection on $[0, 1]$. At this juncture, we remark that it is impossible to answer the question: ``Is a certain level of fairness, specified through any combination or fairness metrics, attainable for a given application?'' through any of the existing models for fairness. To the best of our knowledge, $\varepsilon$-fairness, developed in this article, is the first model that allows us to answer these types of practically relevant questions. 

\begin{table}[htbp]
 \centering

 \begin{tabular}{cc}
 \toprule
 $\varepsilon$ & \# infeasible scenarios \\
 \midrule 
 $0.0 \leqslant \varepsilon \leqslant 0.5$ & 0/9765 \\ 
 $0.6 \leqslant \varepsilon \leqslant 0.7$ & 813/9765 \\ 
 $\varepsilon = 0.8$ & 1443/9765 \\ 
 $\varepsilon = 0.9$ & 1715/9765 \\ 
 $\varepsilon = 1.0$ & 9131/9765 \\
 \bottomrule
 \end{tabular}
 \caption{Number of scenarios for which \eqref{eq:fair-mls} is infeasible for different values of $\varepsilon$.}
 \label{tab:infeasible-q2}
\end{table}

\subsection{Jain et al. index for varying $\varepsilon$}
We now present a plot showing that solutions obtained by enforcing $\varepsilon$-fairness are fair by computing the Jain et al. index of the load shed values across customers for each damage scenario. Fig. \ref{fig:ji-box} shows the box plot of $w(\varepsilon)$ values with $n$ set to 11. One can also observe that the Jain et al. index monotonically increases with increasing $\varepsilon$, indicating that the distribution of load sheds is getting fairer with increasing $\varepsilon$; this serves as validation for the theoretical result in Sec.~\ref{subsec:JI}. We note that the box plot is shown only for the feasible scenarios for all $\varepsilon \in [0, 0.9]$. We excluded $\varepsilon = 1$ as, in this case, only 634 out of the 9756 damage scenarios were feasible (see Table \ref{tab:infeasible-q2}), i.e., for these 634 instances, \eqref{eq:fair-mls} is feasible for all values of $\varepsilon \in [0, 1]$.

\begin{figure}
 \centering
 \includegraphics[scale=0.8]{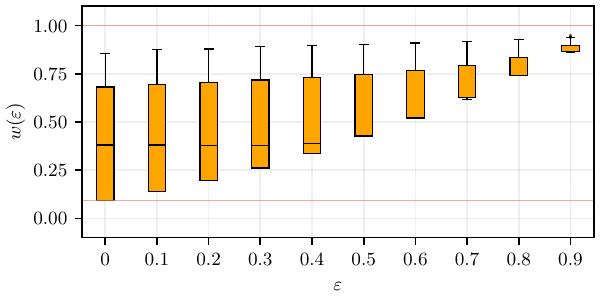}
 \caption{Jain et al. index of the load shed values for varying values of $\varepsilon$. }
 \label{fig:ji-box}
\end{figure}

It was commented in the previous sections that the fairness level, as quantified using the Jain et al. index, is non-monotonic as a function of $p$ for the utility function corresponding to $p$-norm fairness. Fig. \ref{fig:ji-box-p-norm} shows the box plot of the Jain et al. index of the load shed values for varying values of $p$. Similar to Fig. \ref{fig:ji-box}, the box plot is only shown for the 9756 feasible damage scenarios. Specifically, monotonicity fails to hold in 7087 out of the 9756 damage scenarios. 

\begin{figure}
    \centering
    \includegraphics[scale=0.8]{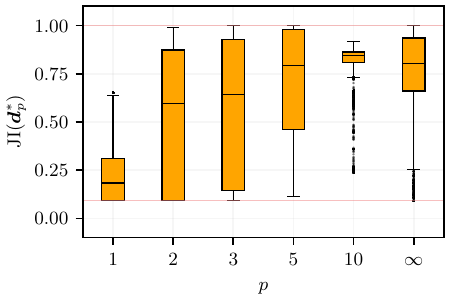}
    \caption{Jain et al. index of the load shed values for varying values of $p$ in Eq. \eqref{eq:p-norm-mls}. Observe that the monotonicity of the index values with respect to $p$ does not hold.}
    \label{fig:ji-box-p-norm}
\end{figure}

\subsection{Efficiency-fairness trade-off}
We now present a systematic approach to quantify the trade-off between efficiency and fairness on the MLS problem enabled by $\varepsilon$-fairness. For completeness, we present the results for the MLS problem where the fairness is enforced using $p$-norm fairness, as in Eq. \eqref{eq:p-norm-mls}. To this end, we first state some obvious facts. For the MLS problem, the maximum achievable efficiency (the minimum amount of load shed) is given by either the optimal objective to \eqref{eq:opt-mls} or the $z(0)$ in \eqref{eq:fair-mls} or $z_0^*$ in \eqref{eq:p-norm-mls}, i.e., when no fairness constraints are imposed. 

When $\varepsilon > 0$ in \eqref{eq:fair-mls}, fairness constraints are enforced and result in a \emph{monotonic} decrease in efficiency (see Proposition \ref{prop:monotonicity}) or, correspondingly, an increase in the minimum load shed from $z(0)$ to $z(\varepsilon)$. We define this reduction in efficiency as a function of $\varepsilon$, denoted by $\eta_r (\varepsilon)$, to be the relative percent increase in load shed when enforcing the distribution of load sheds in $\bm d$ to be at least $\varepsilon$-fair. 
\begin{gather}
 \eta_r(\varepsilon) \triangleq \frac{z(\varepsilon) - z(0)}{z(\varepsilon)} \cdot 100 \% ~\text{ for any $\varepsilon > 0$} \label{eq:n-loss}
\end{gather}

Fig. \ref{fig:n-loss} presents the box plot of $\eta_r(\varepsilon)$ for varying values of $\varepsilon$. It is clear from the plot that the efficiency decreases monotonically with $\varepsilon$, as indicated by Proposition \ref{prop:monotonicity}. In particular, for the MLS problem, the maximum decrease in system efficiency is around $4$\% when $\varepsilon$-fairness is enforced with $\varepsilon = 0.9$. Fig. \ref{fig:n-loss} would be very useful for the transmission system operator to decide what level of fairness is acceptable based on operational efficiency considerations. 

\begin{figure}
 \centering
 \includegraphics[scale=0.8]{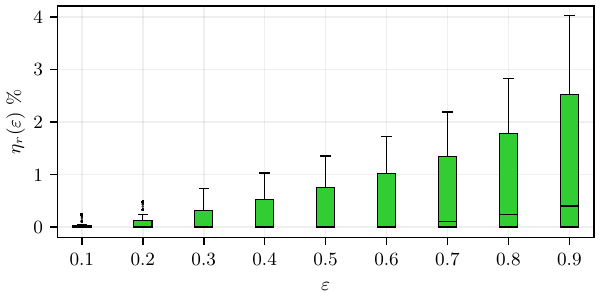}
 \caption{Relative loss in efficiency (\%) as a function of $\varepsilon$.}
 \label{fig:n-loss}
\end{figure}

Now, when $p > 1$ in \eqref{eq:p-norm-mls}, there is no monotonic behavior of the efficiency as a function of $p$ unlike $\varepsilon$ and \eqref{eq:fair-mls}. To show this graphically, we compute 
\begin{gather}
    \eta_r^p \triangleq \frac{z_p^* - z_0^*}{z_p^*} \cdot 100 \% ~\text{ for any $p > 1$} \label{eq:n-loss-p-norm}
\end{gather}
and Fig. \ref{fig:n-loss-p-norm} presents the box plot of $\eta_r^p$ for varying values of $p$. It is evident from the figure that there is no clear relationship between the decrease in efficiency and the value of $p$. A similar trend was observed when the utility function corresponding to the $\alpha$-fairness model, i.e., $U_{\alpha}(\cdot)$ in Eq. \eqref{eq:p-norm-mls} for increasing $\alpha$ values. The absence of a monotonicity property in the utility function-based approaches makes the design of an algorithm to compute the parameter to achieve a desired fairness-efficiency trade-off difficult. 

\begin{figure}
 \centering
 \includegraphics[scale=0.8]{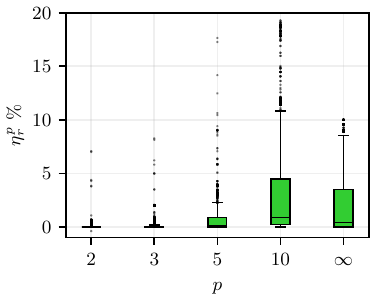}
 \caption{Relative loss in efficiency (\%) as a function of $p$.}
 \label{fig:n-loss-p-norm}
\end{figure}

\section{Conclusion and Future Work} \label{sec:conclusion}
This article develops a parametric model of fairness called $\varepsilon$-fairness that can be directly incorporated into existing optimization problems by adding a single SOC constraint in practical applications. Theoretical properties of $\varepsilon$-fairness that conform with our intuitive understanding of fairness are presented. A case study on incorporating fairness in the optimal operation of a damaged power transmission network also shows how a practitioner can use $\varepsilon$-fairness to quantify the inherent trade-off between fairness and efficiency. 

Future work can proceed in any of the following tangents: (i) study the computational impact of incorporating $\varepsilon$-fairness to combinatorial optimization problems like vehicle routing and facility location, and develop algorithms to address the additional complexity if one exists, (ii) extend the models and algorithms to generalize to the notion of weighted fairness where specific agents are given higher priority or weights and fairness has to be enforced in the distribution of weighted utilities, (iii) develop fast and possibly distributed heuristics to enforce $\varepsilon$-fairness into problems for which similar heuristics to compute good quality fairness-agnostic solutions exist in the literature, and (iv) develop approximation algorithms to incorporate fairness into existing NP-hard decision-making problems.

\section*{Declarations}
\paragraph{Ethics Approval \& Consent to Participate} -- 
Not applicable. 

\paragraph{Consent for Publication} -- Not applicable.

\paragraph{Availability of Data \& Material} -- Not applicable. 

\paragraph{Competing Interests} --
The authors declare that they have no competing interests.

\paragraph{Funding \& Acknowledgements} -- 
The authors acknowledge funding provided by DOE's Grid Modernization Initiative project, ``Power Planning for Alignment of Climate and Energy Systems.'' The research conducted at Los Alamos National Laboratory is done under the auspices of the National Nuclear Security Administration of the US Department of Energy under Contract No. 89233218CNA000001. 

\paragraph{Authors' Contribution} -- Conceptualization, methodology, and manuscript preparation: K.S., D.D., and R.B.

\printbibliography

@inproceedings{taylor2023managing,
  title={Managing Wildfire Risk and Promoting Equity through Optimal Configuration of Networked Microgrids},
  author={Taylor, Sofia and Setyawan, Gabriela and Cui, Bai and Zamzam, Ahmed and Roald, Line A},
  booktitle={Proceedings of the 14th ACM International Conference on Future Energy Systems},
  pages={189--199},
  year={2023}
}

@article{mo2000fair,
  title={Fair end-to-end window-based congestion control},
  author={Mo, Jeonghoon and Walrand, Jean},
  journal={IEEE/ACM Transactions on networking},
  volume={8},
  number={5},
  pages={556--567},
  year={2000},
  publisher={IEEE}
}

@article{hashmi2022can,
  title={Can locational disparity of prosumer energy optimization due to inverter rules be limited?},
  author={Hashmi, Md Umar and Deka, Deepjyoti and Bu{\v{s}}i{\'c}, Ana and Van Hertem, Dirk},
  journal={IEEE Transactions on Power Systems},
  year={2022},
  publisher={IEEE}
}

@article{Gini1936,
  title={On the measure of concentration with special reference to income and statistics},
  author={Gini, Corrado},
  journal={Colorado College Publication, General Series},
  volume={208},
  number={1},
  pages={73--79},
  year={1936},
  publisher={Colorado College, Colorado Springs}
}

@article{Rawls1971,
	title        = {A theory of justice},
	author       = {Rawls, John},
	year         = 1971,
	journal      = {Cambridge (Mass.)}
}

@article{Jain1984,
	title        = {A quantitative measure of fairness and discrimination},
	author       = {Jain, Rajendra K and Chiu, Dah-Ming W and Hawe, William R and others},
	year         = 1984,
	journal      = {Eastern Research Laboratory, Digital Equipment Corporation, Hudson, MA},
	volume       = 21
}

@incollection{Mathieu2023,
	title        = {Algorithms for {E}nergy {J}ustice},
	author       = {Mathieu, Johanna L},
	year         = 2023,
	booktitle    = {Women in Power: Research and Development Advances in Electric Power Systems},
	publisher    = {Springer},
	pages        = {67--83}
}

@article{coffrin2018relaxations,
	title        = {Relaxations of {AC} maximal load delivery for severe contingency analysis},
	author       = {Coffrin, Carleton and Bent, Russell and Tasseff, Byron and Sundar, Kaarthik and Backhaus, Scott},
	year         = 2018,
	journal      = {IEEE Transactions on Power Systems},
	publisher    = {IEEE},
	volume       = 34,
	number       = 2,
	pages        = {1450--1458}
}

@article{huang2023inequalities,
  title={Inequalities across cooling and heating in households: {E}nergy equity gaps},
  author={Huang, Luling and Nock, Destenie and Cong, Shuchen and Qiu, Yueming Lucy},
  journal={Energy Policy},
  volume={182},
  pages={113748},
  year={2023},
  publisher={Elsevier}
}

@book{bertsekas2021data,
  title={Data networks},
  author={Bertsekas, Dimitri and Gallager, Robert},
  year={2021},
  publisher={Athena Scientific}
}

@article{stubbs2009multi,
  title={Multi-portfolio optimization and fairness in allocation of trades},
  author={Stubbs, Robert A and Vandenbussche, Dieter},
  journal={White paper, Axioma Inc. Research Paper},
  volume={13},
  year={2009},
  publisher={Citeseer}
}

@article{bektacs2020using,
  title={Using $\ell$p-norms for fairness in combinatorial optimisation},
  author={Bekta{\c{s}}, Tolga and Letchford, Adam N},
  journal={Computers \& Operations Research},
  volume={120},
  pages={104975},
  year={2020},
  publisher={Elsevier}
}

@article{barbati2016equality,
  title={Equality measures properties for location problems},
  author={Barbati, Maria and Piccolo, Carmela},
  journal={Optimization Letters},
  volume={10},
  pages={903--920},
  year={2016},
  publisher={Springer}
}

@article{matl2018workload,
  title={Workload equity in vehicle routing problems: A survey and analysis},
  author={Matl, Piotr and Hartl, Richard F and Vidal, Thibaut},
  journal={Transportation Science},
  volume={52},
  number={2},
  pages={239--260},
  year={2018},
  publisher={INFORMS}
}

@article{martin2013cooperative,
  title={Cooperative search for fair nurse rosters},
  author={Martin, Simon and Ouelhadj, Djamila and Smet, Pieter and Berghe, Greet Vanden and {\"O}zcan, Ender},
  journal={Expert Systems with Applications},
  volume={40},
  number={16},
  pages={6674--6683},
  year={2013},
  publisher={Elsevier}
}

@book{horn2012matrix,
  title={Matrix analysis},
  author={Horn, Roger A and Johnson, Charles R},
  year={2012},
  publisher={Cambridge university press}
}

@article{pglib,
  title={The power grid library for benchmarking {AC} optimal power flow algorithms},
  author={Babaeinejadsarookolaee, Sogol and Birchfield, Adam and Christie, Richard D and Coffrin, Carleton and DeMarco, Christopher and Diao, Ruisheng and Ferris, Michael and Fliscounakis, Stephane and Greene, Scott and Huang, Renke and others},
  journal={arXiv preprint arXiv:1908.02788},
  year={2019}
}

@article{bezanson2017julia,
  title={Julia: A fresh approach to numerical computing},
  author={Bezanson, Jeff and Edelman, Alan and Karpinski, Stefan and Shah, Viral B},
  journal={SIAM review},
  volume={59},
  number={1},
  pages={65--98},
  year={2017},
  publisher={SIAM}
}

@article{dunning2017jump,
  title={JuMP: A modeling language for mathematical optimization},
  author={Dunning, Iain and Huchette, Joey and Lubin, Miles},
  journal={SIAM review},
  volume={59},
  number={2},
  pages={295--320},
  year={2017},
  publisher={SIAM}
}

@article{cplex,
  title={V22.1.1: {User's Manual for CPLEX}},
  author={CPLEX, IBM ILOG},
  journal={International Business Machines Corporation},
  year={2022}
}

@article{wang2023local,
  title={Local and utility-wide cost allocations for a more equitable wildfire-resilient distribution grid},
  author={Wang, Zhecheng and Wara, Michael and Majumdar, Arun and Rajagopal, Ram},
  journal={Nature Energy},
  pages={1--12},
  year={2023},
  publisher={Nature Publishing Group UK London}
}

@article{shakkottai2008network,
  title={Network optimization and control},
  author={Shakkottai, Srinivas and Srikant, Rayadurgam and others},
  journal={Foundations and Trends{\textregistered} in Networking},
  volume={2},
  number={3},
  pages={271--379},
  year={2008},
  publisher={Now Publishers, Inc.}
}

@article{xinying2023guide,
  title={A guide to formulating fairness in an optimization model},
  author={Xinying Chen, Violet and Hooker, JN},
  journal={Annals of Operations Research},
  pages={1--39},
  year={2023},
  publisher={Springer}
}

@article{nash1950bargaining,
  title={The bargaining problem},
  author={Nash Jr, John F},
  journal={Econometrica: Journal of the econometric society},
  pages={155--162},
  year={1950},
  publisher={JSTOR}
}

@article{kelly1997charging,
  title={Charging and rate control for elastic traffic},
  author={Kelly, Frank},
  journal={European transactions on Telecommunications},
  volume={8},
  number={1},
  pages={33--37},
  year={1997},
  publisher={Wiley Online Library}
}

@article{yu2023alpha,
  title={Alpha-Fair Routing in Urban Air Mobility with Risk-Aware Constraints},
  author={Yu, Yue and Gao, Zhenyu and Li, Sarah HQ and Wei, Qinshuang and Clarke, John-Paul and Topcu, Ufuk},
  journal={arXiv preprint arXiv:2310.00135},
  year={2023}
}

@article{kennedy1996income,
  title={{I}ncome distribution and mortality: cross sectional ecological study of the {R}obin {H}ood index in the {U}nited {S}tates},
  author={Kennedy, Bruce P and Kawachi, Ichiro and Prothrow-Stith, Deborah},
  journal={British Medical Journal},
  volume={312},
  number={7037},
  pages={1004--1007},
  year={1996},
  publisher={British Medical Journal Publishing Group}
}

@article{everitt2010cambridge,
  title={The {C}ambridge dictionary of statistics},
  author={Everitt, Brian S and Skrondal, Anders},
  year={2010}
}

@inproceedings{schwarz2011throughput,
  title={Throughput maximizing multiuser scheduling with adjustable fairness},
  author={Schwarz, Stefan and Mehlfuhrer, Christian and Rupp, Markus},
  booktitle={2011 IEEE International Conference on Communications (ICC)},
  pages={1--5},
  year={2011},
  organization={IEEE}
}

@article{guo2014throughput,
  title={Throughput maximization with short-term and long-term Jain's index constraints in downlink OFDMA systems},
  author={Guo, Chongtao and Sheng, Min and Wang, Xijun and Zhang, Yan},
  journal={IEEE transactions on communications},
  volume={62},
  number={5},
  pages={1503--1517},
  year={2014},
  publisher={IEEE}
}

@article{guo2013jain,
  title={A Jain's index perspective on $\alpha$-fairness resource allocation over slow fading channels},
  author={Guo, Chongtao and Sheng, Min and Zhang, Yan and Wang, Xijun},
  journal={IEEE communications letters},
  volume={17},
  number={4},
  pages={705--708},
  year={2013},
  publisher={IEEE}
}

@article{masoumi2023fairness,
  title={Fairness analysis of indoor multi-user communications through steerable IR-beam},
  author={Masoumi, Hamed and Johar, Masoud and Salehiyan, Alireza and Emadi, Mohammad Javad},
  journal={IET Optoelectronics},
  volume={17},
  number={2-3},
  pages={77--86},
  year={2023},
  publisher={Wiley Online Library}
}

@inproceedings{kumari2023performance,
  title={Performance Enhancement and Scheduling in Communication Networks—A Review into Various Approaches},
  author={Kumari, Priya and Jain, Nitin},
  booktitle={International Conference on Micro-Electronics and Telecommunication Engineering},
  pages={661--672},
  year={2023},
  organization={Springer}
}

\end{document}